\documentclass[reqno]{au2}
\usepackage{amssymb,latexsym,stmaryrd,enumerate}
\usepackage[pdftex]{graphicx}
\usepackage[pdftex]{overpic}
\usepackage{epstopdf}
\newtheorem{theorem}{Theorem}[section]

\newtheorem{lemma}{Lemma}[section]
\def\clap#1{\hbox to 0pt{\hss#1\hss}}

\begin{document}
\title[Identities common to four abelian group operations]
      {Groupoid identities common to four abelian group operations}

\author{David Kelly}
\address{Department of Mathematics\\
         University of Manitoba\\
         Winnipeg, Manitoba, Canada R3T 2N2}
\keywords{finitely based, finite basis, medial groupoid, variety}
\subjclass[2000]{08B05}
\date{July 6, 2008}
\begin{abstract}
We exhibit a finite basis  $\mathcal{M}$  for a certain variety $\mathbf{V}$  of medial
groupoids.  The set  $\mathcal{M}$ consists of the medial law
$(xy)(zt)=(xz)(yt)$  and
five other identities involving four variables.  The
variety $\mathbf{V}$ is generated by the four groupoids  $\pm x\pm y$  on the
integers.  Since $\mathbf{V}$ is a very natural variety, proving it to be
finitely based should be of interest.

In an earlier paper, we made a conjecture which implies that $\mathbf{V}$
is finitely based.  In this paper, we show that $\mathbf{V}$ is finitely based
by proving that  $\mathcal{M}$  is a basis.  Based on our proof,
we think that our conjecture will be difficult to prove.

As we explain in the paper, the variety  $\mathbf{V}$  corresponds to the
Klein $4$-group.  We use this group to show that  $\mathbf{V}$  has a basis
consisting of interchange laws.  (We define ``interchange law''  in
the introduction.)  We give more examples of finite groups where
such a basis exists for the corresponding groupoid variety.  We
also give examples of finite groups where such a basis is
impossible.  The second case is a further challenge to anyone
who tries to prove our conjecture.

We used four medial groupoids to define $\mathbf{V}$.  We also present a
finite basis for the variety generated by any proper subset of
these four groupoids.  In an earlier paper with R.~Padmanabhan,
we gave the corresponding finite bases when the constant zero
is allowed.
\end{abstract}
\maketitle
\setcounter{section}{-1}
\section{Introduction}

 The overview given in the abstract was designed to
motivate the reading of our intricate arguments.
In the next paragraph, we define the sets $\mathcal{M}$
and  $\Sigma$  of identities.  In fact,  $\Sigma$ is the set of identities valid
in the variety $\mathbf{V}$ that was defined in the abstract.
Although it
is ``obvious'' that  $\mathcal{M}$  is a basis for  $\mathbf{V}$, a proof is required.  The
conjecture we made in \cite{dK08}, described later in this introduction, 
implies that  $\mathbf{V}$
is finitely based.
 
 Let  $\Sigma$  be the set of groupoid identities that are satisfied by
the four binary operations  $\pm x\pm y$  in every abelian group.
Theorem 1.1 states that the following six identities form an
independent basis for  $\Sigma$.
\par\textup{(M1)} \  $(xy)(zt)=(xz)(yt)$
\par\textup{(M2)} \  $(xy)(zt)=(ty)(zx)$
\par\textup{(M3)} \ $((xy)z)t=((xt)z)y$
\par\textup{(M4)} \ $(x(yz))t=(x(tz))y$
\par\textup{(M5)} \ $x((yz)t)=z((yx)t)$
\par\textup{(M6)} \ $x(y(zt))=z(y(xt))$
\par\noindent The identity (M1) is called the \emph{medial} \emph{law}.  Let  $\mathcal{M}$  denote the
set of the above six ``mutation laws.''  When the constant zero
is allowed, Kelly and Padmanabhan \cite{KP85} found a finite basis
for the corresponding set of identities.

 When  $G$  is a  multiplicative abelian group generated by $\alpha$
and $\beta$, we write  $\Sigma(G;\alpha,\beta)$  for the set of groupoid identities that
are satisfied in the integral group ring  $\mathbb{Z}[G]$  when the binary
operation is  $\alpha x+\beta y$.  Kelly and Padmanabhan \cite{KP85} showed
that $\Sigma$ equals $\Sigma(\mathbf{K}\mathbf{L};\alpha,\beta)$, where  $\mathbf{K}\mathbf{L}=\{\mspace{1mu}\alpha,\beta,\gamma,1\mspace{1mu}\}$  is the Klein
$4$-group.  Our result for $\Sigma$  supports the conjecture of \cite{dK08}
that $\Sigma(G;\alpha,\beta)$ is finitely based whenever $G$  is finite.

 A term is \emph{linear} when no variable occurs more than once.
If  $p$ is a linear term and we interchange two variables in  $p$ to
form  $q$,  then $p=q$  is an \emph{interchange law}.  Observe that each
identity in $\mathcal{M}$  is an interchange law.

 We present finite bases for the identities satisfied by any
proper subset of the four abelian group operations   $\pm x\pm y$.   All
these bases are shown in Table 1 of \S2.  (When the constant
zero is allowed, the corresponding finite bases appear in
\cite{KP85}.)  To justify Table 1, four bases must be verified, which
is done in Sections 3, 5, 6 and 7.  Sections 5 and 6 each require
a technical result from \S4.

 For finite $G$, Theorem 2.2 of \cite{dK08} characterizes when
$\Sigma(G;\alpha,\beta)$ has a basis consisting of interchange laws.  For certain
finite groups---including the Klein $4$-group---Theorem 9.2
simplifies this characterization.  The final two sections of the
paper concern this new characterization.

   An identity is \emph{balanced}  when each variable occurs
equally often on each side.   Any set of balanced identities is
called \emph{balanced}.   An identity is \emph{linear} if it is balanced and each
side is linear.  We allow $G$  to be an arbitrary $2$-generated
abelian group.  (In Sections  1 to 7, $G$  is always the Klein
$4$-group.)  Each identity $p=q$ in  $\Sigma(G;\alpha,\beta)$  is balanced.  Each
identity in  $\Sigma(G;\alpha,\beta)$  can be obtained by identifying variables in
a linear identity that is in $\Sigma(G;\alpha,\beta)$.  Thus, the linear identities of
$\Sigma(G;\alpha,\beta)$  form a basis for $\Sigma(G;\alpha,\beta)$.

 A \emph{tree} always means a full binary tree, i.e., a finite rooted
tree (growing downwards) in which each non-leaf has exactly
two children.  Every subterm of a linear term $p$  corresponds to a
vertex of the corresponding tree $P$  and vice-versa.  (An
uppercase letter always denotes the corresponding tree.)  A
variable corresponds to a trivial tree.  The tree for the linear term
$\emph{pq}$ is obtained by substituting the trees $P$ and $Q$  for the leaves of
the two-leaved tree.

 The \emph{rank} of a term is the number of its variable occurrences
and the \emph{rank} of a tree is the number of its leaves.   A  \emph{left edge}
(or $\alpha$-\emph{edge}) of a tree is an edge that descends to a left child.   A
vertex that is not a leaf is called \emph{internal}.

 The \emph{color} of a variable in a linear term is its coefficient in
the polynomial ring over  $\mathbb{Z}[G]$ when the binary operation $xy$ is
replaced by $\alpha x+\beta y$.  We color the vertices of the corresponding
tree with elements of  $G$.  We color the root with the identity
element and then descend the tree; the color for the left child is
$\alpha$  times that of the parent  and, for the right child,  $\beta$  times.  On
the leaves of the tree, this coloring agrees with the coloring of
the variables in the linear term.

   A linear identity $p=q$  is in  $\Sigma(G;\alpha,\beta)$  iff  every variable
has the same color in  $p$  and $q$.  Thus, an interchange law is in
$\Sigma(G;\alpha,\beta)$  exactly when the two interchanged variables have the
same color.  Figure 1 shows the tree for each mutation law.
Black-filled circles correspond to the interchanged variables;
their common color (an element of the Klein $4$-group) is also
shown.

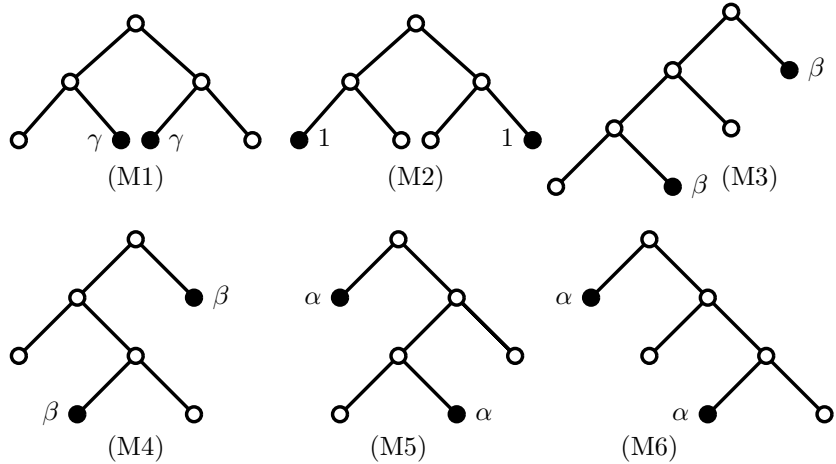
\begin{figure}
\begin{overpic} {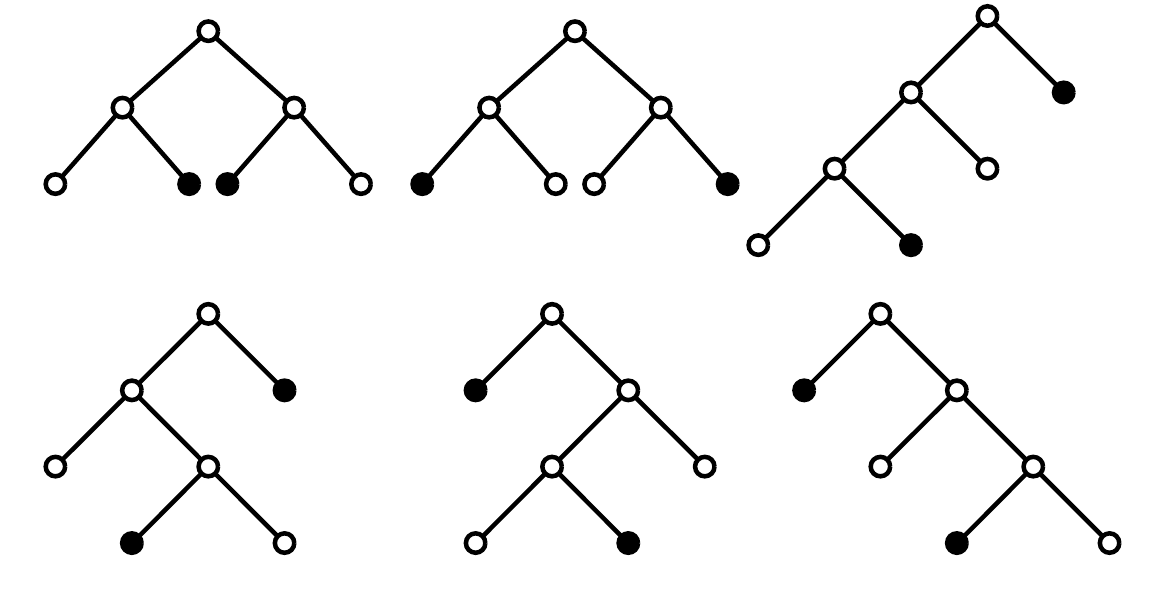}
\put(14.18,35.75){\llap{$\gamma$}}
\put(21.64,35.75){$\gamma$}
\put(17.91,31.4){\clap{(M1)}}
\put(38.39,35.75){$1$}
\put(60.48,35.75){\llap{$1$}}
\put(49.43,31.4){\clap{(M2)}}
\put(93.55,43.63){$\beta$}
\put(80.42,30.5){$\beta$}
\put(86.99,31.4){\clap{(M3)}}
\put(26.57,18.02){$\beta$}
\put(9.25,4.88){\llap{$\beta$}}
\put(17.91,0.94){\clap{(M4)}}
\put(38.81,18.02){\llap{$\alpha$}}
\put(56.12,4.88){$\alpha$}
\put(47.46,0.94){\clap{(M5)}}
\put(67.04,18.02){\llap{$\alpha$}}
\put(80.18,4.88){\llap{$\alpha$}}
\put(75.7,0.94){\clap{(M6)}}
\end{overpic}
\caption{Trees for the mutation laws}
\end{figure}

 Whenever we prove an interchange law from a set of
interchange laws, we stop immediately after successfully
interchanging the two distinguished variables in some derived
term.  Such a proof can be completed by re-applying, in the
reverse order, all the other interchanges that were used.  In any
proof of an interchange law by induction on the rank, we can
stop whenever the two variables are in a proper subterm; we
shall say that the two variables are ``closer.''  We can also stop
when the corresponding two leaves are in a proper subtree.  (By
replacing a suitable internal vertex by a leaf, the original two
leaves are in a tree of lower rank.)

 Let $x$, $r$ and $s$ be vertices of the same color in a tree.  If $x$ is
a leaf,  and $r$ and $s$ are incomparable, then we can replace  $r$
with $x$  by using interchange laws.  The verification is easy.  If  $r$
does not contain  $x$,  then interchange  $r$  and  $x$.  If  $r$  does
contain  $x$,  then first interchange  $r$  and  $s$. This simple
observation is called the ``double rule.''

   The notation $p\equiv q$ means that the terms $p$ and $q$ are
identical.  We write $r\le p$  to indicate that   $r$ is a subterm of  $p$.

 We shall use the ``local'' rule for equational derivation of
McNulty \cite{gM82}.  A \emph{substitution instance} of an identity or a
term is produced by replacing its variables by terms.  We fix a
set of identities $\Pi$ in an arbitrary type and write $p\sim q$ when the
term $q$  is the result of replacing one occurrence of the subterm $r$
in $p$ by the term $s$,  where  $r=s$ or its opposite is a substitution
instance of an identity in $\Pi$. The identity $p=q$ is a consequence of
$\Pi$ iff there is a sequence $p\equiv p_1\sim p_2\sim\dots\sim p_n\equiv q$ for some $n\ge1$.

 Each term $p$ has a \emph{dual} $\widetilde{p}$, obtained by replacing the
groupoid operation by its opposite.  Forming the dual
interchanges the colors $\alpha$ and $\beta$.  The \emph{dual} of an identity $p=q$ is
the identity $\widetilde{p}=\widetilde{q}$.  A set of identities that is closed under duality
is called \emph{self-dual}.  In particular, $\Sigma=\Sigma(\mathbf{K}\mathbf{L};\alpha,\beta)$ is self-dual.  The
\emph{dual} of a tree is its mirror image.  Henceforth, colors are
elements of the Klein $4$-group.

\section{Independent finite basis for  $\Sigma$}

 Let $S$ be the semigroup with 1 that is freely generated by
the ``letters'' $\alpha$ and  $\beta$.  For each $\sigma\in S$, we define (inductively) a
linear term $\overline{\sigma}x$  in the variable $x$  and the \emph{auxiliary variables}
$z_1$, $z_2$, $z_3$,  \dots .  For $\sigma\in S$, we write $|\sigma|$ for its length.  We begin by
defining $\overline{1}x\equiv x$.  For $\sigma\in S$,  $\overline{\alpha\sigma}x\equiv(\overline{\sigma}x)z_{|\sigma |+1}$  and  $\overline{\beta\sigma}x\equiv z_{|\sigma |+1}(\overline{\sigma}x)$.
Observe that the auxiliary variables are numbered beginning at
the maximum depth.  An example is $\overline{\beta\beta\alpha}x\equiv z_3(z_2(xz_1))$.   (This
definition is from \cite{dK08}.)

 Following \cite{dK08}, the \emph{signature} of a descending path from
$u$ to $v$ in a tree is the product in  $S$ (from left to right) of the edge
labels $(\alpha$ or $\beta)$ starting at $u$. We allow $u$  and  $v$  to be equal (in
which case, 1 is the signature).  In the tree for the linear term
$\overline{\sigma}x$,  the path to $x$ has signature $\sigma$.  (When the initial vertex is
unspecified, it is understood to be the root.)  If there is a
descending path in a tree with signature $\sigma$,  then $\sigma$-\emph{terminator} is
our name for final vertex of this path.

 In this section, we call a signature \emph{compressed} when it is
compressed modulo $\mathcal{M}$  in the sense of \cite{dK08}. A signature is
not compressed exactly when two vertices of the same color in
the tree for  $\overline{\sigma}x$ can be interchanged (using  $\mathcal{M}$) so that the new
tree has a shorter path to $x$.  Of course, one of the interchanged
vertices must be an auxiliary variable.

\begin{lemma}
The compressed signatures modulo $\mathcal{M}$  are
$\alpha^k$, $\beta^k$, $\alpha\beta^k$ and $\beta\alpha^k$  for $k\ge0$.
\end{lemma}

\begin{proof}
In the tree for  $\overline{\alpha^k}x$, the internal vertices and $x$  have
color 1 or $\alpha$, while each auxiliary variable has color $\beta$ or $\gamma$. In the
tree for  $\overline{\alpha\beta^k}x$,  the internal vertices and $x$ have color $\alpha$ or $\gamma$,
while each auxiliary variable has color  1 or $\beta$.   Thus, by duality,
all the given signatures are compressed.

 In  $\overline{\alpha^2\beta}x$  or $\overline{\alpha\beta\alpha}x$, the variables $x$ and $z_3$ can be
interchanged by (M3) or (M4).  Therefore, by duality, the
semigroup subterms  $\alpha^2\beta$, $\alpha\beta\alpha$, $\beta^2\alpha$  and  $\beta\alpha\beta$ must be
excluded.  The listed signatures are exactly the ones that remain.
\end{proof}

 The following lemma is a special case of Theorem 9.2.  We
shall give a proof that only uses the characterization theorem of
\cite{dK08}.  The matrix in the following proof is explained in  \S9,
where we shall also calculate---in a very simple way---its
determinant.

\begin{lemma}
The interchange laws form a basis for  $\Sigma$.
\end{lemma}

\begin{proof}
Since the following matrix is nonsingular, the
interchange laws form a basis for  $\Sigma$  by Theorem 2.2 of \cite{dK08}.
\begin{equation*}
\left[
\begin{array}{rrrr}
-1  &1  &1  &0\\
 1 &-1  &0  &1\\
 1  &0 &-1  &1\\
 0  &1  &1 &-1
\end{array}
\right]  \qedhere
\end{equation*}
\end{proof}

\begin{theorem}
The set $\mathcal{M}$ is an independent basis for  $\Sigma$.
\end{theorem}

\begin{proof}
By Lemma 1.2, the interchange laws form a basis
for  $\Sigma$.  Therefore, it suffices to derive each interchange law from
$\mathcal{M}$ .

 Let  $x$ and $y$ be distinct variables of the same color $c$  in the
linear term  $p$.  We can assume that  $p\equiv qr$ , with $x\le q$  and $y\le r$.
By induction on the rank of $p$, we shall show that  $\mathcal{M}$  allows us
to interchange  $x$ and $y$ in $p$.

  Let  $\sigma$ be the signature of the path from $q$  to  $x$ and let  $\tau$
be the signature of the path from $r$   to $y$.  By induction, we can
assume that $q\equiv\overline{\sigma}x$, $r\equiv\overline{\tau}y$, and that both $\sigma$ and $\tau$ are compressed.
(New auxiliary variables are used in $\overline{\tau}y$.)  We shall consider the
four possible values for $c$.  For each value of  $c$,  Lemma 1.1
determines the possible values for $\sigma$ and $\tau$, subject to the
condition that $\alpha\sigma$ and $\beta\tau$  both evaluate to $c$  in  
$\mathbf{K}\mathbf{L}$.

 If  $\sigma=\alpha^k$  for $k\ge2$, then interchange $r$  and the
$\alpha^2\beta$-terminator by (M3) to bring $x$ and $y$ closer.  If  $\sigma=\beta^k$  for  $k\ge2$,
then interchange $r$  and the $\alpha\beta\alpha$-terminator by (M4) to bring $x$
and $y$ closer.  Therefore, we can assume that $k$ is 0 or 1
whenever  $\sigma=\alpha^k$ or  $\sigma=\beta^k$.  By duality,   $k$ is 0 or 1 whenever
$\tau=\alpha^k$ or  $\tau=\beta^k$.  We call the procedures of this paragraph
``exponent reduction.''

 Let $c=\alpha$.  By exponent reduction, we can assume that $\sigma=1$.
In other words,  $q\equiv x$.  Let  $\tau=\beta\alpha^l$  for odd  $l$.  By (M6),  we can
interchange $x$ and the $\beta^2\alpha$-terminator.  We have either
interchanged $x$ and $y$ or the new value of $\sigma$ is $\alpha^{l-1}$ for $l\ge3$.  In the
latter case, apply exponent reduction.  We can now assume that
$\tau=\alpha\beta^l$  for odd $l$.  By (M5),  interchange $x$ and the
$\beta\alpha\beta$-terminator.  If $l=1$, then we have just interchanged  $x$ and  $y$.
Otherwise, $l\ge3$ and the new value of  $\sigma$ is $\beta^{l-1}$,  so that we are
done by exponent reduction.

  The $c=\beta$   case now follows by duality.  In the two
remaining cases, neither $\sigma$ nor $\tau$ is trivial.  If  $\sigma$ and  $\tau$ begin with
the same letter, then the medial law can be used to bring $x$ and $y$
closer.  Therefore, we can assume that  $\sigma$ and  $\tau$ begin with
different letters.

 Let $c=\gamma$.  We first assume that  $\sigma=\beta\alpha^k$  for even  $k$.
Therefore, $\tau=\alpha\beta^l$  for even  $l$.   Use the medial law to transform
$\sigma$  into $\beta^{l+1}$  and   $\tau$  into $\alpha^{k+1}$.  By exponent reduction, $k=l=0$, so
that we can interchange $x$ and $y$ by the medial law.  The
remaining case is that $\sigma=\beta^k$ and $\tau=\alpha^l$ with $k$ and $l$ odd.    By
exponent reduction, $k=l=1$  and we can apply the medial law to
interchange $x$ and $y$.

 Finally, let $c=1$.  We first assume that  $\sigma=\alpha\beta^k$  for even  $k$.
Therefore, $\tau=\beta\alpha^l$  for even  $l$.   Use (M2) to transform  $\sigma$  into
$\alpha^{l+1}$  and   $\tau$  into $\beta^{k+1}$.  By exponent reduction, $k=l=0$, so that
we can interchange $x$ and $y$ by (M2).  The remaining case is that
$\sigma=\alpha^k$ and $\tau=\beta^l$ with $k$ and $l$ odd.    By exponent reduction,
$k=l=1$ and we can apply (M2)  to interchange $x$ and $y$. This
completes the proof that  $\mathcal{M}$  is a basis.

 We now show that  $\mathcal{M}$  is independent.  We consider local
derivations using  $\mathcal{M}$  without one of its identities.  Without (M1),
$\{\mspace{1mu}(ux)(yu)\mspace{1mu}\}$ is closed.  Without (M2), $\{\mspace{1mu}(xu)(uy)\mspace{1mu}\}$ is closed.  Let
$p=q$ be one of the four remaining identities.  Since each side of
every identity in $\mathcal{M}$ has rank 4, no other identity in $\mathcal{M}$  can be
used in a local derivation of $p=q$.  (See Figure 1.)
\end{proof}

\section{Subsets of abelian group operations}

 We write the four abelian group operations as follows:
$f_1(x,y)=x+y$,  $f_2(x,y)=x-y$,  $f_3(x,y)=-x+y$ and  $f_4(x,y)=-x-y$.  For
any proper subset  $K$  of $\{\mspace{1mu}1,2,3,4\mspace{1mu}\}$, we write   $\Sigma_K$  for the
groupoid identities that are satisfied in $\mathbb{Z}$ by   $f_k$  for every $k\in K$.
If $1\in K$,  then $\Sigma_K$ is balanced because it is a subset of  $\Sigma_1$.  On
the other hand, $\Sigma_{2,3,4}$  is not balanced.

 Table 1 gives a finite basis for every   $\Sigma_K$.  Observe that
duality interchanges $f_2$ and $f_3$.  The source for each basis is also
given in the table.  Up to duality, there are four new results in
the table.

\begin{table} \begin{center} \begin{tabular}{c l l}
\hline
 $K$  &Basis for $\Sigma_K$   &Reference  \\ \hline
1 &$x(yz)=(xy)z$, $xy=yx$ &folklore  \\
2 &$x(y(z(xy)))=z$ &Tarski \cite{aT38}  \\
3 &$(((yx)z)y)x=z$ &duality  \\
4 &(M1), $xy=yx$, $x(xy)=y$ &Je\v{z}ek and Kepka \cite{JK83}  \\
$1, 2$ &(M1), $(xy)z=(xz)y$, $x(zy)=y(zx)$ &Kelly \cite{dK08} \\
$1, 3$ &(M1), $z(yx)=y(zx)$, $(yz)x=(xz)y$ &duality  \\
$1, 4$ &(M1), $xy=yx$, $x(z(ty))=y(z(tx))$ &Kelly \cite{dK08}  \\
$2, 3$ &(M1), $x^2=y^{2}$, $(xx^2)x^2=x$ &Kelly \&{} Padmanabhan \cite{KP85}  \\
$2, 4$ &(M2), $x(xy)=y$ &Theorem 3.2  \\
$3, 4$ &(M2), $(yx)x=y$ &duality  \\
$1,2,3$ &$\mathcal{M}$, $(x^2y)z^2=(z^2y)x^2$, $(xy^2)z^2=(xy^2)z^2$ &Theorem 5.1  \\
$1,2,4$ &$\mathcal{M}$,  $x(x(yz))=(x(zy))x$ &Theorem 6.1  \\
$1,3,4$ &$\mathcal{M}$, $((zy)x)x=x((yz)x)$ &duality  \\
$2,3,4$ &(M1), (M2), $(xy^2)y^2=x$ &Theorem 7.2  \\ \hline \\
 
\end{tabular} 
\caption{Finite bases for all selections of abelian group operations} 
\end{center} \end{table}

 In fact, Gr\"atzer and Padmanabhan \cite{GP78} proved that  $\Sigma_{2,3}$
is one-based.  Padmanabhan \cite{rP69} determined all the terms $p$
of rank five such that $\{\mspace{1mu}p=x\mspace{1mu}\}$ is a basis for $\Sigma_2$; moreover, he
showed that five is the minimum rank for a term  $p$ so that
$\{\mspace{1mu}p=x\mspace{1mu}\}$ is a basis for $\Sigma_2$.

 For a term $p$  in variables  $X$,  we write $[p]$ for its value  $\sum($
$a_xx\mid x\in X)$ when the product $xy$  is replaced by $\alpha x+\beta y$.  In
particular, $p=q$ is in $\Sigma$  iff $[p]=[q]$.  Observe that each coefficient
$a_x$  is in $\mathbb{N}[\mathbf{K}\mathbf{L}]$ and that only finitely many coefficients are
nonzero.  Our proof of the following result uses \cite{KP85}, where
the constant zero was allowed.

\begin{lemma}
In the following three statements, each $n_x$  is a suitable integer.
For terms $p$ and $q$,
\begin{enumerate}[\rm(i)]
\item $p=q$ is in $\Sigma_{1,2,3}$ iff \ $[p]-[q]=\sum(n_x(\alpha+\beta-\gamma-1)x\mid x\in X)$.
\item $p=q$ is in $\Sigma_{1,2,4}$ iff \ $[p]-[q]=\sum(n_x(\alpha-\beta+\gamma-1)x\mid x\in X)$.
\item $p=q$ is in $\Sigma_{2,3,4}$ iff \ $[p]-[q]=\sum(n_x(\alpha+\beta+\gamma+1)x\mid x\in X)$.
\end{enumerate}
\end{lemma}

\begin{proof}
We require some results from Table 2 of  \cite{KP85}.
If $p=q$ is in $\Sigma_{1,2,3}$, then from that table,
$[p]-[q]=\sum(r_x(\alpha+\beta-\gamma-1)x\mid x\in X)$, where each $r_x$ is in $\mathbb{Z}[\mathbf{K}\mathbf{L}]$.
Since
$(a\alpha+b\beta+c\gamma+d)(\alpha+\beta-\gamma-1)=(-a-b+c+d)(\alpha+\beta-\gamma-1)$, condition (i)
follows.  The argument is similar for the other two cases.
\end{proof}

 Later, we shall apply the following immediate consequence
of Lemma 2.1.

\begin{lemma}
For each of $\Sigma_{1,2,3}$ and  $\Sigma_{1,2,4}$,  a basis consists
of the identities $p=q$ that satisfy the following conditions.  The
symbol $x$  denotes any variable that occurs in $p$ or $q$.  If $x$ occurs
exactly once in both $p$ and $q$, then the color of $x$ is the same in $p$
and $q$.  Whenever $x$  does not occur exactly once in $p$  and $q$,
then $x$ occurs exactly  twice in both $p$ and  $q$.  Moreover, when $x$
occurs twice, then:
\begin{enumerate}[\rm(i)]
\item for $\Sigma_{1,2,3}$, it occurs with colors $\alpha$ and $\beta$ in one term and with
colors $\gamma$ and $1$ in the other;
\item for $\Sigma_{1,2,4}$, it occurs with colors $\alpha$ and $\gamma$ in one term and with
colors $\beta$ and $1$ in the other.
\end{enumerate}
\end{lemma}

 Let us call the identities described in Lemma 2.2
\emph{general identities}.  Thus, $\Sigma_{1,2,3}$ and  $\Sigma_{1,2,4}$ each has a basis consisting of
general identities (where the meaning of ``general'' depends on
the context).  Any identity in each of these two sets can be
obtained by identifying the variables in some general identity.

\section{Independent finite basis for the operations  $x-y$  and $-x-y$}

 Let  $\mathcal{B}_{2,4}=\mathcal{M}\cup\{\mspace{1mu}x(xy)=y\mspace{1mu}\}$.  For  $\Sigma_{2,4}$,  we first show that $\mathcal{B}_{2,4}$
is a basis and we then find an independent basis.  We require the
following result, which is similar to Lemma 2.1.

\begin{lemma}
For terms $p$ and $q$,  $p=q$ is in $\Sigma_{2,4}$ iff
$[p]-[q]=\sum((m_x(\alpha+\gamma)+n_x(\beta+1))x\mid x\in X)$ for integers $m_x$ and  $n_x$.
\end{lemma}

\begin{proof}
By Table 2 of  \cite{KP85},  $p=q$ is in $\Sigma_{2,4}$  iff
$[p]-[q]=\sum(r_x(\beta+1)x\mid x\in X)$, where each $r_x$ is in $\mathbb{Z}[\mathbf{K}\mathbf{L}]$.  Since
$(a\alpha+b\beta+c\gamma+d)(\beta+1)=(a+c)(\alpha+\gamma)+(b+d)(\beta+1)$, the result follows.
\end{proof}

\begin{theorem}
The set  $\mathcal{B}_{2,4}$ is a basis for  $\Sigma_{2,4}$.
\end{theorem}

\begin{proof}
Since  $\mathcal{M}$  is a basis for $\Sigma$  by Theorem 1.1, we can
calculate modulo $\Sigma$.  Let  $p=q$ be in $\Sigma_{2,4}$.  Let  $t$  be a fixed
variable. If a variable $x$ occurs in $p$ with colors $\alpha$  and $\gamma$, then
form $p'$ by replacing these two occurrences of  $x$  with $t$.  Since
$t(tp)=x(xp')$ is in $\Sigma$,  it follows that    $p=p'$.  Using (M4),
$(x(yz))x=(x(xz))y=zy$.   Consequently,  $(xy)x=(y(x(xy)))y=y^2y$.
Therefore,  $(xy)x=(zy)z$.   If $x$ has colors $\beta$  and 1 in $p$ and we
form $p'$ by replacing these two occurrences of $x$ with $t$, then
$p=p'$ because $p\equiv p_1p_2=(t(p_2p_1))t=(x(p'_2p'_1))x=p'$.  Make all
possible such variable replacements in both  $p$  and  $q$.

 By Lemma 3.1,  $[p]-[q]=(m(\alpha+\gamma)+n(\beta+1))t$ \ for integers $m$
and  $n$.
Let  $r$  be any term.  If we define $r_i$ by  $r_0\equiv r$ and $r_{i+1}\equiv t(tr_i)$, then
$r=r_i$  is a consequence of $\mathcal{B}_{2,4}$.  If $m<0$, replace $p$ by $p_{|m|}$,   and if
$m>0$, replace $q$ by $q_m$.  Thus, we have reduced to the case that
$m=0$.  If either  $p$ or $q$ is a variable, then replace $p$ by $t(tp)$ and
replace  $q$  by $t(tq)$.  We now define $(uv)^*\equiv(t(vu))t$ for terms $u$
and $v$.  We define a new sequence of terms:   $r_0\equiv r$ and $r_{i+1}\equiv(r_i)^*$.
Clearly,  $r=r_i$  is a consequence of $\mathcal{B}_{2,4}$.   If $n<0$, replace $p$ by $p_{|n|}$,
and if $n>0$, replace $q$ by $q_n$.  Since we have reduced to the case
that $n=0$, the transformed identity is in $\Sigma$ and we are done.
\end{proof}

\begin{theorem}
An independent basis for  $\Sigma_{2,4}$  consists of
\textup{(M2)} and $x(xy)=y$.
\end{theorem}

\begin{proof}
Any consequence of (M2) is balanced.  Moreover,
the second projection satisfies $x(xy)=y$, but fails (M2) in a 
$2$-element set.  Thus, the two identities are independent.

 Assume both (M2) and $x(xy)=y$.  By Theorem 3.1, it
suffices to derive the five remaining identities of  $\mathcal{M}$.  The second
identity allows us to cancel on the left.  By (M2),
$(xy)((zy)z)=(zy)((zy)x)=x$.  Thus, $(xy)((xy)x)=(xy)((zy)z)$ and
we can cancel on the left to conclude that $(xy)x=(zy)z$.
Calculating,
$y^2(x^2y^2)=y^2(yx)^2=((yx)y)^2=(x^2x)^2=x^2(x^2x^2)=x^2$.  Since
$y^2(y^2x^2)=x^2$, the identity  $x^2y^2=y^2x^2$  follows by left cancellation.
Calculating,
$(xy)((xz)(yt))=((yt)y)((xz)x)=(t^2t)(z^2z)=(zt)(z^2t^2)=(zt)(t^2z^2)=(zt)(zt)^2=zt$.
Since $(xy)((xy)(zt))=zt$, we conclude, by left cancellation, that
(M1) holds.

 We can now use both (M1) and (M2).  The calculation
$(x(yz))t=(x(yz))(tt^2)=(t^2(yz))(tx)=((ty)(tz))(tx)=(x(tz))(t(ty))=(x(tz))y$ proves (M3). The calculation
$(x(yz))t=(x(yz))(tt^2)=(t^2(yz))(tx)=((ty)(tz))(tx)=(x(tz))(t(ty))=(x(tz))y$  proves (M4).
The calculation
$x((yz)t)=(xx^2)((yz)t)=(tx^2)((yz)x)=(tx^2)((yz)(xx^2))=(tx^2)((yx)(zx^2))=((zx^2)x^2)((yx)t)=((xx^2)(xz))((yx)t)=(x(xz))((yx)t)=z((yx)t)$
proves (M5).  The calculation
$x(y(zt))=(xx^2)(y(zt))=((zt)x^2)(yx)=((xt)(xz))(yx)=(x(xz))(y(xt))=z(y(xt))$ proves (M6).
\end{proof}

\section{Two results about terms}

 We shall apply the results of this section in \S5 and \S6.

 Let $p$  be a term in which no variable occurs more than
once with the same color.  There is an obvious tree $P$  associated
with $p$ (which extends the definition given for linear terms).  The
variable \emph{occurrences}  in $p$  now correspond to the leaves of  $p$.
For example, if  a variable $x$  occurs with colors $\alpha$ and $\beta$ in  $p$,
then the $\alpha$-leaf $x$ and the $\beta$-leaf $x$  are two distinct leaves of $P$.
As before, the subterms of  $p$  uniquely correspond to subtrees of
$P$.

 For a tree  $T$, let $g^\#$ denote the number of its  $g$-leaves,
where $g\in\mathbf{K}\mathbf{L}$.  We set   $\lambda(T)=(\alpha^\#,\beta^\#,\gamma^\#,1^\#)$ and we call  $\lambda(T)$ the
\emph{total color} of $T$.  A $4$-tuple of natural numbers is called
\emph{representable} if it equals   $\lambda(T)$ for some tree $T$. We first
characterize the representable $4$-tuples.  The second result of this
section concerns subterms.

 Let  $m$ and $n$ be nonnegative integers. We define two
functions: $\varphi_1(m,n)=(2m+n-1)/3$  and  $\varphi_2(m,n)=(m+2n-2)/3$.
Each function returns an integer when  $m$ and  $n$ satisfy
$2m+n\equiv1 \pmod{3}$.  We also define these two functions on
$4$-tuples by defining  $\varphi_i(a,b,c,d)$  to be $\varphi_i(a+b,c+d)$.   For a tree  $T$,
we write  $\varphi_i(T)$ for  $\varphi_i(\lambda(T))$.  By the following result,  $\varphi_1(T)$
and $\varphi_2(T)$  are nonnegative integers for any tree $T$.

\begin{theorem}
\begin{enumerate}[\rm(i)]
\item If $\lambda(T)=(a,b,c,d)$ for a tree $T$,  then
$2m+n\equiv1\pmod{3}$, where $m=a+b$ and $n=c+d$.
\item Every tree  $T$ has   $\,\varphi_1(T)$ \ $\alpha$-vertices,  $\,\varphi_1(T)$ \ $\beta$-vertices,  $\varphi_2(T)$ \ 
$\gamma$-vertices and \\
$(\varphi_2(T)+1)$ \ $1$-vertices.
 \item In any nontrivial tree $T$,  $\alpha^\#\le\varphi_1(T)$,  $\beta^\#\le\varphi_1(T)$,  $\gamma^\#\le\varphi_2(T)$
and  \\ $1^\#\le\varphi_2(T)$.
\item A $4$-tuple $(a,b,c,d)$ of nonnegative integers different than
$(0,0,0,1)$ is representable iff  $2a+2b+c+d\equiv1\pmod{3}$,
$a\le\varphi_1(a,b,c,d)$,  $b\le\varphi_1(a,b,c,d)$,  $c\le\varphi_2(a,b,c,d)$  and  $d\le\varphi_2(a,b,c,d)$.
\end{enumerate}
\end{theorem}

\begin{proof}
We prove (i) and (ii) simultaneously by induction
on the rank of $T$.  Both statements hold for the trivial tree which
has the total color $(0,0,0,1)$.  We can now assume that  $T$  is
nontrivial.  Consider a pair of sibling leaves at the maximum
depth in $T$  and assume the result for the tree $S$  obtained by
removing these two leaves.  Let  $\lambda(S)=(a',b',c',d')$.  If the
maximum depth is odd, then $a=a'+1$, $b=b'+1$ and  $c+d=c'+d'-1$.
Thus, (i) holds for $T$,  $\varphi_1(T)=\varphi_1(S)+1$ and $\varphi_2(T)=\varphi_2(S)$.  Since $T$
has one more $\alpha$-vertex and one more $\beta$-vertex than $S$, condition
(ii) holds for $T$.  The proof for even maximum depth is similar.

 Condition (iii) follows immediately from (ii).  Moreover, the
necessity in condition (iv) follows from (i) and (iii).   Let $(a,b,c,d)$
be a $4$-tuple of nonnegative integers different than $(0,0,0,1)$ that
satisfies the five parts of condition (iv).  Let $m=a+b$ and $n=c+d$.
Thus, $2m+n\equiv1 \pmod{3}$.  Observe that
$\varphi_1(m,n)+\varphi_2(m,n)=m+n-1$.  By induction on  $m+n$, we show that
there is a tree  $T$  with  $\lambda(T)=(a,b,c,d)$.  If there is a nontrivial
tree with total color $(a,b,c,d)$, then there are also trees with total
colors $(b,a,c,d)$, $(b,a,d,c)$ and $(a,b,d,c)$.  (Apply the dual in the
first case and replace the tree for  $pq$ by the tree for  $\emph{qp}$ in the
second.)  Therefore, we can assume that $a\le b$ and $c\le d$.  We shall
write $\varphi_i$  for $\varphi_i(m,n)$.

 If $a$ and $c$ were both zero, then $m+n=b+d\le\varphi_1+\varphi_2=m+n-1$, a
contradiction.  We first assume that $c>0$.  If $a=\varphi_1$, then $a=b$  and
$3b=4b+c+d-1$, which is impossible because $c+d\ge2$.  Therefore,
$a<\varphi_1$.  Since $(a+1,b,c-1,d-1)$ satisfies (iv), we are done by
induction.  (In the representing tree, replace an $\alpha$-leaf by the tree
of rank 2.)  We can now assume that $a>0$.  If $c=\varphi_2$, then $c=d$
and  $3d=a+b+4d-2$, implying that $(a,b,c,d)$ equals $(1,1,0,0)$, the
total color of the tree of rank 2.  Thus, we can assume that $c<\varphi_2$.
Since  $(a-1,b-1,c+1,d)$ satisfies (iv), it is representable.  Replace
a $\gamma$-leaf to complete the proof.
\end{proof}

\begin{theorem}
If a linear term contains variables of all
four colors, then modulo $\Sigma$, the term has  a subterm of the form
$((xy)v)(zt)$ or $(u(yx))(zt)$,  where $x$, $y$, $z$, $t$, $u$ are variables and $v$ is
a term.
\end{theorem}

\begin{proof}
Assume that variables $x$, $y$, $z$ and $t$  occur in the
linear term $p$  with colors $\alpha$, $\beta$, $\gamma$ and 1, respectively.  All
calculations with identities are modulo $\Sigma$.  We induct on the rank
of $p$.  We first assume there are no internal vertices of color $\alpha$  in the tree
$P$.  In particular, $p=xq$.  Since $Q$ contains a $\beta$-leaf, it also contains
an $\alpha$-vertex, say $u$.  (All colors are calculated in $P$.)
By our assumption, $u$  is a leaf.  Interchange  $x$ and the variable
$u$ so that $p=ur$.  The variables $x$, $y$, $z$ and $t$ occur in the term $r$
with colors $\gamma$, 1, $\beta$ and $\alpha$, respectively.  Thus, by induction, $r$
has a subterm with one of the two given forms and we are done.
By duality, we can now assume that $P$  has internal vertices of
color $\alpha$ and of color  $\beta$.

 By the double rule, we can assume that  $p\equiv(qr)(zt)$ for
terms $q$ and $r$.   We first assume that $q\equiv q_1q_2$ and $r\equiv r_1r_2$.  Using
(M1) and (M2), we can assume that $x\le r$ and $y\le r$.  Interchange $x$
and $q_1$, and $y$ and $q_2$ to give a term of the first form.

 We now assume that the tree $P$ has no internal $\gamma$-vertices.  In
particular, $r$ is a variable. As before, $q\equiv q_1q_2$.   We are done if
both $q_1$  and $q_2$ are variables.  Firstly, we assume that $q_1$ is not a
variable.  In $P$, descend from $q_1$ by $\alpha$-edges until we reach a leaf
$u$. If $u$ has color 1,  then interchange the parent of $u$ and the leaf
$x$.  (Since there are no internal $\gamma$-vertices, the sibling of $u$ is a
leaf.)  Hence, we can assume that  $u\equiv x$.  Let  $v$  be the sibling of
$x$.  Since  $v$ and $q_2$ both have color $\beta$, we can use the double rule
to replace $v$ with $y$.  Thus, $((xy)s)w$  is now a subterm.
Interchange  $w$  and $zt$ to complete the proof in this case.
Secondly, we can assume that $q_1\equiv x$  and that $q_2$ is not a variable.
Descend from $q_2$ by $\beta$-edges until we reach a leaf  $u$.  Arguing as
before, we can assume that $u\equiv y$.   Interchange $x$ and the sibling
of  $y$  to make  $xy$ a subterm.  Now interchange $q$ and $t$.
Interchange $r$ and $z$ to make  $tz$ a subterm.  Thus, there is a
subterm  $w(s(xy))$.   Interchange $w$  and $tz$  to obtain $(tz)(s(xy))$
as a subterm.  Since $(tz)(s(xy))=(ts)(z(xy))=((xy)s)(zt)$, we are
done in this case.

 Finally, we can assume that the tree $P$ has no internal $1$-vertices.  In
particular, $q$ is a variable.  Similarly as before, $r\equiv r_1r_2$ and we are
done if both $r_1$  and $r_2$ are variables.  Firstly, we assume that  $r_1$
is not a variable and we descend from  $r_1$ by $\alpha$-edges until we
come to a leaf $u$.  As before, we can assume that $u\equiv y$.  Since $r_2$
has color $\alpha$, the double rule allows us to assume that  $yx$  is a
subterm.  Interchange $q$ and $t$ , and also $r$  and $z$.  In particular,
$\emph{tz}$ is now a subterm.  We now have the subterm  $((yx)s)w$.
Interchange $w$  and $tz$  to obtain $((yx)s)(tz)$ as a subterm.
Secondly, we can assume that $r_1\equiv y$  and that $r_2$ is not a variable.
Descend from $r_2$ by $\beta$-edges until we come to a leaf $u$, which as
before, we can assume is  $x$.  Now make $yx$ a subterm by
interchanging $y$  and the sibling of  $x$.  Thus, there is a subterm
$(w(s(yx))$.  Interchange $w$  and $zt$  to obtain $(zt)(s(yx))$ as a
subterm.  Since  $(zt)(s(yx))=(zs)(t(yx))=((yx)s)(tz)$, the proof is
complete.
\end{proof}

\section{Independent finite basis for all operations except $-x-y$}

 Let
$\mathcal{B}_{1,2,3}=\mathcal{M}\cup\{\mspace{1mu}(x^2z)y^2=(y^2z)x^2,\:(zx^2)y^2=(zy^2)x^2\mspace{1mu}\}$.
We shall show that $\mathcal{B}_{1,2,3}$   is a basis for $\Sigma_{1,2,3}$.

\begin{lemma}
If  $\psi(x,y)$  is a term in which the variable $x$
occurs once with color $\alpha$ and once with color $\beta$, and the variable
$y$ occurs once with color $\gamma$ and once with $1$, then the identity
$\psi(x,y)=\psi(y,x)$  is a consequence of  $\mathcal{B}_{1,2,3}$.
\end{lemma}

\begin{proof}
Let $p\equiv\psi(x,y)$ be the term described above.  We can
assume that $x$ and $y$ each occur exactly twice and that no other
variable occurs more than once in  $p$.    Since  $\mathcal{M}$  is a basis for $\Sigma$
by Theorem 1.1,  we can calculate with identities modulo  $\Sigma$.
By Theorem 4.2, there is a term $q$ such that  $p=q$   modulo  $\Sigma$
and $q$ contains a subterm $r$  of the form  $((x_1x_2)v)(y_1y_2)$ or
$(u(x_2x_1))(y_1y_2)$,  where  $x_1$, $x_2$, $y_1$ and $y_2$  are variables. For later
use, record a sequence of vertex interchanges that takes us from
$P$ to $Q$.   Let  $c$  be the color of the vertex $r$  in  $Q$.  If  $c\in\{\mspace{1mu}\gamma,1\mspace{1mu}\}$,
then use interchanges to replace both  $x_1$ and  $x_2$  by the variable
$x$.  For example, if $c=1$, then interchange the $\alpha$-occurrence of  $x$
and the variable $x_1$ (unless $x$ is already $x_1$).
If $c\in\{\mspace{1mu}\alpha,\beta\mspace{1mu}\}$,  replace
both  $x_1$ and  $x_2$  by $y$ .  In the same way, replace $y_1$ and $y_2$  by $x$
or $y$, as appropriate.   By the suitable identity of  $\mathcal{B}_{1,2,3}$,  we can
interchange the subterms $xx$  and $yy$  in the modified subterm $r$.
Now re-do the all the previous interchanges in the reverse order
to obtain  $\psi(y,x)$.
\end{proof}

\begin{theorem}
The set  $\mathcal{B}_{1,2,3}$ is an independent basis for
$\Sigma_{1,2,3}$.
\end{theorem}

\begin{proof}
Let  $p=q$  be in $\Sigma_{1,2,3}$.  We shall show that $p=q$
follows from  $\mathcal{B}_{1,2,3}$.    Since $\mathcal{M}$ is a basis for  $\Sigma$ by Theorem 1.1,
we can calculate with identities modulo  $\Sigma$.   We can assume
that $p=q$ satisfies condition (i) of  Lemma 2.2.   Of the variables
that occur twice in $p$,  let $X$ be those that have colors  $\alpha$ and $\beta$,
and let  $Y$  be those that have colors $\gamma$ and  1.  (In $q$, the variables
in $X$  have colors   $\gamma$ and 1, while  the variables in $Y$  have colors
$\alpha$ and $\beta$.)    By symmetry, we can assume that$|X|\le |Y|$.  Let
$f:X\to Y$  be a one-to-one function.

 For each $x\in X$,  Lemma 5.1 shows that both occurrences of
$x$ in $p$  can be exchanged with both occurrences of $f(x)$  in $p$.
Thus, we have reduced to the case that  $X$  is empty.

 If  $Y$  is empty, we are done.  Therefore, we can assume
that the cardinality $n$ of $Y$  is nonzero.   Let  $\lambda(P)=(a,b,c,d)$.
Thus, $\lambda(Q)=(a+n,b+n,c-n,d-n)$.  By Theorem 4.1,
$2(a+b)+c+d\equiv2(a+b+2n)+c+d-2n\equiv1 \pmod{3}$.  Therefore,  $n\equiv0
\pmod{3}$.   Let $n=3k$ with $k>0$.

 From  $a+3k\le\varphi_1(Q)=\varphi_1(P)+2k$, it follows that $a+k\le\varphi_1(P)$.
Therefore,
 $a+1\le\varphi_1(P)$,  so that $a+3\le\varphi_1(P)+2$.  Similarly,  $b+3\le\varphi_1(P)+2$.
Since $c\le\varphi_2(P)$  and $d\le\varphi_2(P)$, the sequence $(a+3,b+3,c-3,d-3)$ is
representable by Theorem 4.1.   By this observation and
induction on $k$, we can assume that $n=3$.  In particular,
$\varphi_1(Q)=\varphi_1(P)+2$, $\varphi_2(Q)=\varphi_2(P)-2$,  $c\ge3$ and $d\ge3$.  Let $Y=\{\mspace{1mu}x,y,z\mspace{1mu}\}$.

 Let  $\mathbf{s}=(a+1,b+1,c-2,d-2)$.  Clearly, $\varphi_1(\mathbf{s})=\varphi_1(P)$  and
$\varphi_2(\mathbf{s})=\varphi_2(P)-2$.  Since both $a+1$ and $b+1$ are at most $\varphi_1(P)$ by
the previous paragraph,  $\mathbf{s}$  is representable by Theorem 4.1.  Let
$\psi(t,z)$ be a term of total color $\mathbf{s}$  whose variables are those of $p$
with $x$ and $y$ removed, and $t$ added.  Moreover, each old variable
occurs with the same colors in $p$ and  $\psi(t,z)$.  The new variable $t$
occurs with colors $\alpha$ and $\beta$ in $\psi(t,z)$.  Since $[p]=[\psi(xy,z)]$, the
identity  $p=\psi(xy,z)$  is in $\Sigma$.  By Lemma 5.1,  $\psi(t,z)=\psi(z,t)$ is a
consequence of  $\mathcal{B}_{1,2,3}$.  Substituting, $xy$ for $t$, we obtain
$p=\psi(xy,z)=\psi(z,xy)\equiv r$.  Since  $[r]=[q]$, the identity  $r=q$ is $\Sigma$, and
we have shown that $\mathcal{B}_{1,2,3}$ is a basis.

 We now show that  $\mathcal{B}_{1,2,3}$   is independent. Let $\varepsilon_1$ denote
$(x^2z)y^2=(y^2z)x^2$  and let $\varepsilon_2$ denote $(zx^2)y^2=(zy^2)x^2$.  Since neither
side of $\varepsilon_1$ or $\varepsilon_2$ is linear, these identities cannot be used in a local
derivation of any identity in $\mathcal{M}$.  Thus, since $\mathcal{M}$  is independent by
Theorem 1.1, no identity in  $\mathcal{M}$ can be omitted from  $\mathcal{B}_{1,2,3}$.
Suppose that there is a local derivation of  $\varepsilon_1$ from  $\mathcal{B}_{1,2,3}-\{\mspace{1mu}\varepsilon_1\mspace{1mu}\}$.
Since $\varepsilon_1$   is not in  $\Sigma$,  the identity $\varepsilon_2$  must be used in this
derivation;  let $p\sim q$  be the first time that $\varepsilon_2$  was used.
Therefore, $\lambda(P)=(1,1,2,1)$.  Since   $p$  and $\varepsilon_2$  both have rank 5,
$p$ is a substitution instance of $(zx^2)y^2$  in which $x$, $y$ and $z$ are
replaced by variables.  Hence, $\lambda(P)=(1,1,1,2)$, a contradiction.
Therefore,   $\varepsilon_1$ cannot be omitted.  Similarly, $\varepsilon_2$ cannot be
omitted.
\end{proof}

\section{Independent finite basis for all operations except $-x+y$}

 Let  $\mathcal{B}_{1,2,4}=\mathcal{M}\cup\{\mspace{1mu}x(x(yz))=(x(zy))x\mspace{1mu}\}$.  We shall show that
$\mathcal{B}_{1,2,4}$   is a basis for $\Sigma_{1,2,4}$.

\begin{lemma}
Both $((xy)z)(xy)=((yx)z)(yx)$ and
$(z(yx))(xy)=(z(xy))(yx)$ are consequences of $\mathcal{B}_{1,2,4}$.
\end{lemma}

\begin{proof}
Let $\varepsilon$ denote  $x(x(yz))=(x(zy))x$.  Since
$((xy)z)(xy)=((x(zy))x)y$  is in  $\Sigma$,  we can derive it from  $\mathcal{M}$  (by
Theorem 1.1).  The identity $((x(zy))x)y=(x(x(yz)))y\equiv p$  is a
consequence of $\varepsilon$.  Since $x$ and $y$ occur with the same colors in
$p$, the identity $p=(y(y(xz)))x$  is in $\Sigma$.  Thus, we have derived
$((xy)z)(xy)=(y(y(xz)))x$.  By interchanging $x$ and $y$ in this
identity, we obtain $((yx)z)(yx)=p$.  Hence,
$((xy)z)(xy)=((yx)z)(yx)$.

 We now give the argument for the second identity.  The
identity $(z(yx))(xy)=x((y(zx))y)$ is in $\Sigma$.  Consequently,
$x((y(zx))y)=x(y(y(xz)))\equiv q$ using  $\varepsilon$.  Since the identity
$q=y(x(x(yz)))$  is in $\Sigma$, we have derived $(z(yx))(xy)=y(x(x(yz)))$.
Interchange $x$ and $y$ to obtain  $(z(xy))(yx)=q$.  Hence,
$(z(yx))(xy)=(z(xy))(yx)$.
\end{proof}

\begin{lemma}
If  $\psi(x,y)$  is a term in which the variable $x$
occurs with colors $\alpha$ and $\gamma$, and the variable $y$ occurs with colors
$\beta$ and 1, then the identity $\psi(x,y)=\psi(y,x)$  is a consequence of
$\mathcal{B}_{1,2,4}$.
\end{lemma}

\begin{proof}
We use the two identities of Lemma 6.1.  The rest
of the proof is a slight modification of the proof of Lemma 5.1.
In this proof, we use condition (ii) of Lemma 2.2 and we
interchange  the  subterms $xy$ and $yx$ (rather
than $xx$ and $yy$).  Also, the two pairs of colors are now $\{\mspace{1mu}\alpha,\gamma\mspace{1mu}\}$ and
$\{\mspace{1mu}\beta,1\mspace{1mu}\}$. (Multiplication by any element of the Klein $4$-group
permutes these two sets.)
\end{proof}

\begin{theorem}
The set  $\mathcal{B}_{1,2,4}$ is an independent basis for
$\Sigma_{1,2,4}$.
\end{theorem}

\begin{proof}
We write $\varepsilon$ for the identity  $x(x(yz))=(x(zy))x$.  Let
$p=q$  be in $\Sigma_{1,2,4}$.  We shall show that $p=q$ follows from  $\mathcal{B}_{1,2,4}$.
Since $\mathcal{M}$ is a basis for  $\Sigma$ by Theorem 1.1, we can calculate with
identities modulo  $\Sigma$.   We can assume that $p=q$ satisfies
condition (ii) of  Lemma 2.2.   Of the variables that occur exactly
twice in $p$,  let $X$ be those having colors  $\alpha$ and $\gamma$, and let  $Y$  be
those having colors $\beta$ and  1.  (In $q$, the variables in $X$  have
colors   $\beta$ and 1, while  the variables in $Y$  have colors  $\alpha$ and $\gamma$.)
By symmetry, we can assume that $|X|\le |Y|$.  Let $f:X\to Y$  be a
one-to-one function.

 For each $x\in X$,  Lemma 6.2 shows that both occurrences of
$x$ in $p$  can be exchanged with both occurrences of $f(x)$  in $p$.
Thus, we have reduced to the case that  $X$  is empty.

 If  $Y$  is empty, we are done.  Therefore, we can assume
that the cardinality $n$ of $Y$  is nonzero.   Let  $\lambda(P)=(a,b,c,d)$.
Thus, $\lambda(Q)=(a+n,b-n,c+n,d-n)$.  For $i=1,2$, let us write $\varphi_i$ for
the common values of  $\varphi_i(P)$ and $\varphi_i(Q)$.  Observe that
$\varphi_i(a,b-1,c,d-1)=\varphi_i-1$.   Since $(0,1,0,2)$ is not representable, the
$4$-tuple $(a, b-1,c,d-1)$ is not $(0,0,0,1)$.  From Theorem 4.1,
$a+n\le\varphi_1$  and $c+n\le\varphi_2$.  Thus, by Theorem 4.1, $(a, b-1,c,d-1)$ is
representable.  Let $rs$ be a term whose total color is
$(a, b-1,c,d-1)$.  We impose additional conditions on the variables
in the term $\emph{rs}$.   Choose some $x\in Y$ and let the variables of $rs$ be
those of $p$  without  $x$.  Moreover, each remaining variable
occurs exactly as many times in $rs$  as it does in $p$,  and with the
same colors.  Consequently, $[p]=[(x(sr))x]$,  which means that
$p=(x(sr))x$  is in $\Sigma$.  By $\varepsilon$,  $(x(sr))x=x(x(rs))$.  Thus, we have
derived $p=p'\equiv x(x(rs))$, where
$\lambda(P')=(a+1,b-1,c+1,d-1)$.  Observe that the variable $x$  occurs
with the same colors in $p'$ and $q$.  By induction on $n$,  we have
shown that $\mathcal{B}_{1,2,4}$ is a basis.

 We now show that  $\mathcal{B}_{1,2,4}$   is independent.  As in the proof
of Theorem 5.1, no identity in  $\mathcal{M}$ can be omitted.  Since $\varepsilon$  is not
in $\Sigma$, it can also not be omitted.
\end{proof}

\section{Finite basis for all operations except $x+y$}

\begin{theorem}
The set  $\mathcal{B}_{2,3,4}=\mathcal{M}\cup\{\mspace{1mu}(xy^2)y^2=x\mspace{1mu}\}$ is a basis
for  $\Sigma_{2,3,4}$.
\end{theorem}

\begin{proof}
Let  $p=q$  be in $\Sigma_{2,3,4}$.  We shall show that $p=q$
follows from  $\mathcal{B}_{2,3,4}$.  Recall that $\mathcal{M}$ is a basis for  $\Sigma$ by Theorem
1.1.  Let  $t$  be a fixed variable.  If a variable $x$  occurs with all
four colors in $p$,  then the calculation $p=(pt^2)t^2=(p'x^2)x^2=p'$
shows that we can replace these four occurrences of  $x$  by $t$.
(The term $p'$ is the term $p$ with this replacement; the identity
$(pt^2)t^2=(p'x^2)x^2$  is in $\Sigma$.)  Repeat this replacement as often as
possible on both $p$  and $q$.  Thus, we can assume that $t$  is the
only variable that occurs with all four colors in either $p$ or $q$.

 By Lemma 2.1,   $[p]-[q]=n(\alpha+\beta+\gamma+1)t$ for some integer $n$.
By symmetry, we can assume that $n\ge0$.  Repeat the operation
$r\mapsto(rt^2)t^2$ $n$  times on  $q$  to obtain $q'$.  Since $p=q'$ is in $\Sigma$, we
have shown that $\mathcal{B}_{2,3,4}$  is a basis.
\end{proof}

\begin{theorem}
The set
$\mathcal{B}=\{\mspace{1mu}\textup{(M1)},\:\textup{(M2)},\:(xy^2)y^2=x\mspace{1mu}\}$
 is a basis for  $\Sigma_{2,3,4}$.
\end{theorem}

\begin{proof}
By Theorem 7.1, it suffices to derive (M3) to (M6)
from  $\mathcal{B}$.  It is easy to derive $y^2(y^2x)=x$, the dual of $(xy^2)y^2=x$,
from  $\mathcal{B}$.  Thus,   we can apply duality.  Consequently, it suffices to
derive (M3) and (M4).

 Using $\mathcal{B}$,
$((xy)z)t=((xy)z)((tu^2)u^2)=((xy)(tu^2))(zu^2)=((xt)(yu^2))(zu^2)$,
so that (M3) is a consequence.  Similarly,
$(x(yz))t=(x(yz))t=(x(yz))(u^2(u^2t))=(xu^2)((yz)(u^2t))=(xu^2)((tz)(u^2y))$,
so that (M4) is a consequence.
\end{proof}

\section{Multicirculant matrices}

 In the next section, we shall apply Theorem 8.1 below.
This theorem is due to P.J. Davis (see \S5.8 of \cite{pD79}).  We
include an elementary proof of Davis's result.

 For $k\ge1$  and a sequence  $\mathbf{s}=(s_1, s_2,\dots, s_k)$ of positive
integers,  let
$\mathcal{G}(\mathbf{s})=S_1\times S_2\times\dots\times S_n$,
where for $1\le i\le k$, $S_i$ is the additive group of
integers modulo $s_i$.  Let $n=s_1s_2\dots s_k$.  We also define a bijection
from the group  $\mathcal{G}(\mathbf{s})$ onto the set
$\{\mspace{1mu}0,1,2, \dots ,n-1\mspace{1mu}\}$ by
\begin{equation*}
   (x_1, x_2, x_3,\dots, x_k)^*=x_1+x_2s_1+x_3s_1s_2+\dots+x_k(s_1s_2\dots s_{k-1}).
\end{equation*}
Observe that $(0,0, \dots , 0)^*=0$.

 We shall define an $n\times n$ matrix  $\mathcal{M}(\mathbf{s})=[a_{i,j}]$,  where $0\le i,j<n$.
The top row (when $i=0$) is arbitrary.  For nonzero $i=(x_1, x_2,  \dots , x_k)^*$
and any  $j=(y_1, y_2,\dots, y_k)^*$,
$a_{i,j}=a_{0,t}$ where $t=(y_1-x_1, y_2-x_2,\dots, y_k-x_k)^*$.  We call  $\mathcal{M}(\mathbf{s})$ a
\emph{multicirculant} matrix of \emph{level}  $k$.

 Let  $G$ be a finite abelian group, written additively.  For
$g\in G$, let  $\chi_g$  be the character associated with $g$.   It is well
known that  $\sum(\chi_g(h)\mid h\in G)$ equals $|G|$ when $g=0$, and equals zero for
every other $g$.

\begin{theorem}
For  $k\ge1$, let   $\mathbf{s}=(s_1, s_2,\dots, s_k)$ be a
sequence of positive integers whose product is $n$.  For each
$\mathbf{x}=(x_1, x_2,  \dots , x_k)\in\mathcal{G}(\mathbf{s})$,
let $c_\mathbf{x}$ be a complex number.  Let  $A=\mathcal{M}(\mathbf{s})$ be
the multicirculant matrix  of level $k$ that is defined by setting
$a_{0,\mathbf{x}^*}=c_\mathbf{x}$  for every  $\mathbf{x}\in\mathcal{G}(\mathbf{s})$.  The eigenvalues of $A$ (including
multiplicities) are
$\sum(c_\mathbf{x}\xi^{x_1}_1\xi^{x_2}_2\cdots\xi^{x_k}_k\mid\mathbf{x}\in\mathcal{G}(\mathbf{s}))$
as each $\xi_i$ runs over
all $s_i$-th roots of unity.
\end{theorem}

\begin{proof}
 Let $\lambda=\sum(c_\mathbf{x}\xi^{x_1}_1\xi^{x_2}_2\cdots\xi^{x_k}_k\mid\mathbf{x}
\in\mathcal{G}(\mathbf{s}))$, where  $\xi_i$ is an $s_i$-th root of unity for  $1\le i\le k$.  Also, let $\mathbf{v}=(v_0, v_1,\dots, v_{n-1})$,
where $v_{\mathbf{x}^*}=\xi^{x_1}_1\xi^{x_2}_2\cdots\xi^{x_k}_k$.  We first show that $\mathbf{v}$ is an eigenvector
for $\lambda$.  The dot product of row  $\mathbf{y}^{*}$ of
$A$ and the vector  $\mathbf{v}$  is
\begin{align*}
  \sum&(c_{\mathbf{x}-\mathbf{y}}\xi^{x_1}_1\xi^{x_2}_2\cdots\xi^{x_k}_k\mid\mathbf{x}\in\mathcal{G}(\mathbf{s}))\\ 
=&\sum(c_\mathbf{z}\xi^{y_1+z_1}_1\xi^{y_2+z_2}_2\cdots\xi^{y_k+z_k}_k\mid\mathbf{z}\in\mathcal{G}(\mathbf{s}))\\
=&\left(\sum(c_\mathbf{z}\xi^{z_1}_1\xi^{z_2}_2\cdots\xi^{z_k}_k\mid\mathbf{z}\in\mathcal{G}(\mathbf{s}))\right)\xi^{y_1}_1\xi^{y_2}_2\cdots\xi^{y_k}_k=\lambda v_{\mathbf{y}^*}.
\end{align*}

 Let  $\omega_i$ ($1\le i\le k$) be a primitive $s_i$-th  root of unity. For any
$\mathbf{y}=(y_1, y_2,\dots, y_k)$, let $\xi_i=\omega^{y_i}_i$  and define the eigenvector  $\mathbf{v}$  as above.  Consequently, the
component  $\mathbf{x}^{*}$  of   $\mathbf{v}$  equals  $\chi_\mathbf{y}(\mathbf{x})$.   Let  $P$  be the matrix
whose rows are the eigenvectors  $\mathbf{v}$, indexed by $\mathbf{y}^{*}$ for
$\mathbf{y}\in\mathcal{G}(\mathbf{s})$,  and let   $Q$  be the matrix whose columns are the same
eigenvectors, but indexed by $(-\mathbf{z}\mathbf{)}^{*}$ for $\mathbf{z}\in\mathcal{G}(\mathbf{s})$.  The $(\mathbf{y}^{*},\mathbf{z}^{*})$-entry
of   $\emph{PQ} $  is $\sum(\chi_{(\mathbf{y}-\mathbf{z})^*}(\mathbf{x})\mid\mathbf{x}\in\mathcal{G}(\mathbf{s}))$.  By the well-known result
mentioned above, $PQ=nI$.  Hence, $P$  is nonsingular, which
means that the eigenvectors are linearly independent.
\end{proof}

\section{Basis of interchange laws}

  Theorem 9.2 below is a significant generalization of Lemma 1.2.

 Let  $G$  be a finite abelian group generated by $\alpha$ and  $\beta$.
In  particular, $G$  is isomorphic to the direct product of two cyclic
groups.  We recall some notation from \cite{dK08}.   A \emph{vector}  means a
function from $G$ to the integers.  For each  $g\in G$, let  $\mathbf{e}_{g}$  be the
vector that is 1 at $g$ and 0 elsewhere.   For each  $g\in G$, we define
$\mathbf{v}_{g}=-\mathbf{e}_{g}+\mathbf{e}_{\alpha g}+\mathbf{e}_{\beta g}$.  By Theorem 2.2 of \cite{dK08}, the following two
conditions are equivalent:
\begin{itemize}
\item A basis for $\Sigma(G;\alpha,\beta)$ consists of its interchange laws.
\item The set  $\{\mspace{1mu}\mathbf{v}_{g}\mid g\in G\mspace{1mu}\}$ is linearly independent.
\end{itemize}

\begin{theorem}
Let $G$ be  $\langle\delta\rangle\times\langle\varepsilon\rangle$, the direct product of
cyclic groups of orders $m$ and $n$, respectively, and assume that
$\alpha=\delta^a\varepsilon^{a'}$  and $\beta=\delta^b\varepsilon^{b'}$.  The interchange laws form a basis for
$\Sigma(G;\alpha,\beta)$ iff $-1+\omega^a\xi^{a'}+\omega^b\xi^{b'}$  is never zero whenever  $\omega$ is an $m$-th
root of unity and  $\xi$  is an $n$-th root of unity.
\end{theorem}

\begin{proof}
The group $G$ is isomorpic to  $\mathcal{G}(m,n)$, where
we convert to addition and replace 1 by 0.  Consequently,
$\mathbf{a}=(a,a')$ and  $\mathbf{b}=(b,b')$  are the images of $\alpha$ and $\beta$, respectively.
We use our bijection and index our vectors by  $0, 1,  \dots , mn-1$
rather than by $G$.  Thus, for $g\in \mathcal{G}(m,n)$,  $\mathbf{e}_{g}$ is now the vector that is 1
at $g^*$  and 0 elsewhere.  For $g\in \mathcal{G}(m,n)$, the definition of   $\mathbf{v}_{g}$  is now
$\mathbf{v}_{g}=-\mathbf{e}_{g}+\mathbf{e}_{a+g}+\mathbf{e}_{b+g}$.

 Let $\mathbf{v}_{0}$  be the top row of the multicirculant matrix
$A=\mathcal{M}(m,n)$.   Observe that the rows of  $A$  are the vectors  $\mathbf{v}_{g}$  for
$g$ in $\mathcal{G}(m,n)$.  By Theorem 8.1, each eigenvalue of  $A$ equals $-1+\omega^a\xi^{a'}
+\omega^b\xi^{b'}$,  where  $\omega$ is an $m$th root of unity and  $\xi$  is an  $n$th root
of unity.  Thus,  $\{\mspace{1mu}\mathbf{v}_{g}\mid g\in \mathcal{G}(m,n)\mspace{1mu}\}$ is linearly independent iff
$-1+\omega^a\xi^{a'}+\omega^b\xi^{b'}$  is never zero when $\omega$  and  $\xi$  are as previously specified.
Now apply Theorem 2.2 of \cite{dK08}, which we described above.
\end{proof}

\begin{theorem}
Let $G=\langle\alpha\rangle\oplus\langle\beta\rangle$, a direct sum, where $\alpha$
has order $m$ and $\beta$ has order $n$.  The interchange laws form a
basis for $\Sigma(G;\alpha,\beta)$ iff $m$ and $n$ are not both multiples of $6$.
\end{theorem}

\begin{proof}
By Theorem 9.1, we must determine when
$-1+\omega+\xi=0$  for an $m$th root of unity $\omega$  and an $n$th root of unity
$\xi$.  Since the absolute values of the imaginary parts of  $\omega$ and $\xi$
are equal, so are the absolute values of their real parts.  Thus, $1/2$
is the real part of both $\omega$ and $\xi$.  The result now follows.
\end{proof}

 When $m=n=2$ in Theorem 9.2, we obtain  Lemma  1.2.  Let
$A$  be the matrix in the proof of Lemma 1.2.  Clearly,  $A=\mathcal{M}(2,2)$
and its top row is as in the proof of Theorem 9.1 when $\mathbf{a}=(1,0)$
and  $\mathbf{b}=(0,1)$.  Since the eigenvalues of  $A$ are $1,-1,-1,-3$ by
Theorem 8.1, the determinant of $A$ is $-3$.

 For a finite \emph{cyclic} group $G$, Theorem 3.1 of
\cite{dK08} determines exactly when the interchange laws form a basis
for  $\Sigma(G;\alpha,\beta)$.

\end{document}